\documentclass[11pt]{article}

\usepackage{amsmath,amssymb,amsfonts,textcomp,amsthm,xifthen,psfrag,graphicx,color,MnSymbol}
\usepackage[T1]{fontenc}

\date{}

\newtheorem{theorem}{Theorem}
\newtheorem{lemma}[theorem]{Lemma}
\newtheorem{cor}[theorem]{Corollary}
\newtheorem{prop}[theorem]{Proposition}

\theoremstyle{definition} 

\newcommand{\<}{\langle{}}
\renewcommand{\>}{\rangle}

\def\pwnabla{\nabla_{\cT_n}}
\def\curl{{\rm curl\,}}

\def\div{{\rm div\,}}
\def\pwdiv{ {\rm div}_{\cT_n}\,}

\newcommand{\jump}[1]{[#1]}

\newcommand{\R}{\ensuremath{\mathbb{R}}}
\newcommand{\N}{\ensuremath{\mathbb{N}}}
\newcommand{\HH}{\ensuremath{\mathbf{H}}}

\newcommand{\LL}{\ensuremath{\mathbf{L}}}

\newcommand{\nn}{\ensuremath{\mathbf{n}}}
\newcommand{\qq}{\ensuremath{\boldsymbol{q}}}
\newcommand{\vv}{\ensuremath{\boldsymbol{v}}}
\newcommand{\ww}{\ensuremath{\boldsymbol{w}}}
\newcommand{\zz}{\ensuremath{\boldsymbol{z}}}

\newcommand{\cT}{\ensuremath{\mathcal{T}}}

\newcommand{\cS}{\ensuremath{\mathcal{S}}}
\newcommand{\cR}{\ensuremath{\mathcal{R}}}

\newcommand{\cB}{\ensuremath{\mathcal{B}}}

\newcommand{\FF}{\ensuremath{\mathbf{F}}}

\newcommand{\ssigma}{{\boldsymbol\sigma}}
\newcommand{\ttau}{{\boldsymbol\tau}}
\renewcommand{\qq}{{\boldsymbol{q}}}
\newcommand{\uu}{\boldsymbol{u}}

\newcounter{constantsnumber}
\def\setc#1{
  \ifthenelse{\equal{#1}{poinc}}{C_{\rm edge}}{ 
   \refstepcounter{constantsnumber}
   \label{const#1}C_{\theconstantsnumber}}}

\def\c#1{
  \ifthenelse{\equal{#1}{poinc}}{C_{\rm edge}}{ 
    C_{\ref{const#1}}}}

\title{A time-stepping DPG scheme for the heat equation
\thanks{Supported by CONICYT through FONDECYT projects 1150056, 3150012,
        and Anillo ACT1118 (ANANUM).}}
\author{
Thomas~F\"uhrer$^\dagger$
\and
Norbert Heuer$^\dagger$
\and
Jhuma Sen Gupta\thanks{
Facultad de Matem\'aticas, Pontificia Universidad Cat\'olica de Chile,
Avenida Vicu\~na Mackenna 4860, Santiago, Chile,
email: {\tt \{tofuhrer,nheuer,jsengupta\}@mat.uc.cl}}}

\begin{document}
\maketitle
\begin{abstract}
We introduce and analyze a discontinuous Petrov-Galerkin method with optimal test functions
for the heat equation. The scheme is based on the backward Euler time stepping and uses an ultra-weak
variational formulation at each time step. We prove the stability of the method for
the field variables (the original unknown and its gradient weighted by the square root
of the time step) and derive a C\'ea-type error estimate. For low-order approximation spaces
this implies certain convergence orders when time steps are not too small in comparison with mesh sizes.
Some numerical experiments are reported to support our theoretical results.

\bigskip
\noindent
{\em Key words}:
heat equation, parabolic problem, DPG method with optimal test functions,
least squares method, ultra-weak formulation, backward Euler scheme, Rothe's method

\noindent
{\em AMS Subject Classification}: 65M60, 65M12, 65M15
\end{abstract}

\section{Introduction}

In recent years the discontinuous Petrov-Galerkin method with optimal test functions (DPG method)
has been established as a discretization technique that can deliver robust error control
for singularly perturbed problems. This includes convection-dominated diffusion
\cite{DemkowiczH_13_RDM,NiemiCC_13_ASD,ChanHBTD_14_RDM,BroersenS_14_RPG} and 
reaction-dominated diffusion \cite{HeuerK_RDM}, the latter also in combination with 
boundary elements \cite{FuehrerH_RCD}. In this paper we propose an extension of
the DPG framework to the heat equation. Indeed, there is no need to analyze yet another
method for this equation. Numerical methods for parabolic problems are well established,
see, e.g., Thom\'ee's book \cite{Thomee_06_GFE}. However, our aim is to
robustly discretize singularly perturbed parabolic problems. In this paper we propose
a framework that can be easily extended to such problems, though a proof of robustness
is not necessarily straightforward and left to future research.

A natural way to approximate parabolic problems by the DPG method is to apply its
framework in a time-domain setting. The DPG discretization will be robust in the
energy norm and, having chosen a variational formulation, it ``only'' remains to check
which variables are controlled by the energy norm and in which norms.
This is the strategy that Ellis, Chan and Demkowicz \cite{EllisCD_RDM} persued to
deal with time-dependent convection-dominated diffusion. Their analysis proves a robust
control of the primary field variable in $L^2$, but a control of its gradient is
not guaranteed. In this paper we apply a classical time-stepping approach, combined with
the DPG method, to gain control also of this second field variable.

Standard methodology for parabolic problems is the \emph{method of lines}, i.e.,
to transform them into systems of ordinary differential equations, e.g.,
by using a Galerkin discretization, and to solve them by using time-stepping methods
based on finite differences. Theoretical studies usually also follow these steps:
analyze the semi-discrete setting and the fully discrete scheme.
By design, DPG analysis is based on the continuous setting and this implies stability
of the discretization and convergence. Discrete features enter the analysis only when
taking into account the fact that optimal test functions have to be approximated
(we do not deal with this effect here). Therefore, combining time-stepping with
the DPG method does not require to treat the fully discrete scheme differently than
the semi-discrete one. In this setting, semi-discrete refers to the system that is
continuous in space and discrete in time. First discretizing in time and then
in space is also known as Rothe's method \cite{Rothe_30_ZPR}.
It has the advantage over the method of lines that space discretization
can be chosen independently for different time steps in a straightforward way.
In fact, our analysis does not require that approximating spaces at different
time steps are related.
The viewpoint of Rothe's method is relevant when considering small time steps
which introduce a singular perturbation and can generate spurious space oscillations,
cf., e.g., \cite{Bornemann_90_AMA,BochevGS_04_ISS,Harari_04_SSF}.
In fact, our DPG scheme does not require stabilization as the methods in
\cite{BochevGS_04_ISS,Harari_04_SSF}. Our choice of test functions guarantees
(in the ideal case) that errors are controlled in a robust way.

We discretize the heat equation by the backward Euler method in time. We then obtain
a family of formally singularly perturbed reaction-diffusion problems, one for each time step.
These problems are solved robustly by a DPG scheme, i.e., we have to take the
time-step parameter into account both for the setting and in the analysis.
There is a catch with this approach that makes its analysis tricky. Having
an approximation at each time step (which should be robust), the effect of these
approximations accumulate over time so that stability in the final time
(uniformly in the time steps) is not guaranteed in general. This stability is
achievable only when showing step-wise stability with constant \emph{one}
(or a constant that tends to $1$ sufficiently fast). Time-stepping schemes
are prone to error accumulation over time and therefore, standard elliptic analysis
at each time step is not sufficient to control the whole setting.

The situation is different when using Bubnov-Galerkin approximations in space.
In those cases there is space symmetry between the ansatz and test sides so that
the analysis can take into account the parabolic equation including time derivative.
For instance, testing the error equation with the error $e$ produces the term
$(\dot e,e)=\frac 12 \frac d{dt} (e,e)$ which can be integrated in time.
In the Petrov-Galerkin setting this trick is not applicable, at least not in this
simple form.

The DPG method can be interpreted as a least-squares or minimum residual method.
In fact, there are proposals to use least-squares approaches in space combined
with time stepping for parabolic problems. Bramble and Thom\'ee \cite{BrambleT_72_SLS}
have analyzed this combination already in 1972. As in our case,
their analysis faces the challenge to control the accumulation of space errors over time.
They proposed to compensate this by high-order space approximations and using
splines of higher regularity than necessary for standard Galerkin approximations.
In this paper, we derive estimates that are independent of any particular
choice of discrete ansatz spaces. As mentioned, DPG analysis is done at the
continuous level and implies stability and convergence of the discrete scheme
(of course, under the assumption of sufficient regularity and approximation
properties of discrete spaces). It is then left to decide whether one wants to use
higher order approximations or sufficiently fine space meshes to control the
time accumulation of space errors.

Let us note that a different strategy is to use least-squares both in space and time
\cite{MajidiS_01_LSG,MajidiS_02_LSG}. Then one automatically gains control of
the error in space and time, but sacrifices the simplicity and efficiency of time-stepping
approaches. It would be interesting to study such a least-squares strategy in time,
combined with DPG in space. But here we analyze a simple time-stepping procedure and
the objective is two-fold. On the one hand, we are interested in a general approach
that in principle extends to singularly perturbed problems. On the other hand,
we want to provide an analysis for the heat equation that is as sharp as possible.
Since there is not much flexibility in proving stability of time-stepping schemes,
we expect our techniques to be useful for parabolic problems beyond the heat equation.

In the next section we present the model problem, introduce the time discretization
and DPG setting, and state our main results. Theorem~\ref{thm_stab} states the stability
of our fully discrete scheme and Theorem~\ref{thm_Cea} gives an abstract
error estimate. In Corollary~\ref{cor_conv} we indicate the convergence orders
for the case of lower order approximations. In Section~\ref{sec_anal} we collect 
some technical results (stability of adjoint problems and norm equivalences in
the ansatz space) and prove the main theorems. Some numerical experiments are
reported in Section~\ref{sec_num}.

Throughout the paper, except for Corollary~\ref{cor_conv}, there are no generic
or unspecified constants.

\section{Model problem and DPG scheme}

Let $\Omega\subset\R^d$ ($d=2,3$) be a bounded simply connected Lipschitz domain with boundary
$\Gamma=\partial\Omega$. Our model problem is the heat equation
\begin{subequations} \label{heat}
\begin{alignat}{2}
    \dot u - \Delta u &= f    &\qquad& \text{in} \quad \Omega\times (0,T], \label{heata} \\
    u                 &= 0    && \text{on} \quad \Gamma\times (0,T], \label{heatb}\\
    u                 &= u_0  && \text{in} \quad \Omega\times \{0\}. \label{heatc}
\end{alignat}
\end{subequations}
Here, $\dot u:=\frac{\partial u}{\partial t}$, and we will assume that the initial datum $u_0$
and forcing term $f$ are smooth enough (so that all the norms of $u$ in this paper do exist).

In the following, we present our time-stepping DPG scheme and state all the main results.
In \S\ref{sec_21} we start with introducing some notation and standard function spaces.
We also define our time steps and space meshes, some of the function spaces depend on.
Section~\ref{sec_22} introduces our backward Euler semi-discrete scheme and recalls its
well-posedness and stability (Proposition~\ref{prop_stab}). The fully discrete scheme is
defined in \S\ref{sec_23}. There, we also state all the main results, stability in
Theorem~\ref{thm_stab}, quasi-optimal convergence in Theorem~\ref{thm_Cea}, and some convergence
properties in Corollary~\ref{cor_conv}.

\subsection{Function spaces and space-time discretization} \label{sec_21}

For a Lebesgue measurable set $\mathcal{M}\subset\R^{d}$ ($d=1,2,3$),
$L^2(\mathcal{M})$, $H^1(\mathcal{M})$, $H_0^1(\mathcal{M})$ and $H^{-1}(\mathcal{M})$
are the standard Sobolev spaces with the usual norms. In particular, we denote the norms
of $L^2(\mathcal{M})$ by $\|\cdot\|_{\mathcal{M}}$ induced by the inner product
$(\cdot, \cdot)_{\mathcal{M}}$ and skip the index $\mathcal{M}$ when $\mathcal{M} = \Omega$.
We will also use the vector valued function spaces $\LL^2(\mathcal{M})$,
$\HH(\div,\mathcal{M})$ and $\HH(\curl,\mathcal{M})$ with the usual norms.
For a real Banach space $\cB$ with norm $\|\cdot\|_\cB$, we also need the space-time function spaces
\[
   L^\infty(0,T;\cB)
   =
   \Bigl\{v:\; (0,T)\rightarrow \cB;\; \sup_{t\in (0,T)} \|v(t)\|_\cB < \infty \Bigr\},
\]
\[
   W^{m,1}(0,T;\cB)
   =
   \Bigl\{v:\; (0,T)\rightarrow \cB;\; \int_0^T \sum_{l=0}^m
                                          \Bigl\|\frac {\partial^l v(t)}{\partial t^l}\Bigr\|_\cB
                                       \,\mathrm{d}t < \infty
   \Bigr\}
   \quad (m\in\N).
\]
Now, let us introduce space-time decompositions.
We use discrete time steps $0=t_0<t_1\cdots<t_N=T$ and denote $k_n:= t_n-t_{n-1}$.
Corresponding to the time steps $t_n$, we consider a family $(\cT_n)_{n=0}^N$
of shape regular conforming partitions of $\bar{\Omega}$ into a set of disjoint
elements $K$ with Lipschitz boundary $\partial K$.
To each mesh $\cT_n$ we associate a skeleton $\cS_n$ that consists of the
boundaries of elements of $\cT_n$ (a precise definition is not required since we will
define the functionals supported there).

Broken Sobolev spaces play an important role in DPG techniques since they are used to localize
the calculation of optimal test functions, cf.~\cite{DemkowiczG_11_ADM}.
In our case, we have the families of broken spaces
\begin{align*}
   H^1(\cT_n)
   &:=
   \{ v \in L^2(\Omega);\; v|_K \in H^1(K) \quad\forall K \in\cT_n\},\\
   \HH(\div,\cT_n)
   &:=
   \{\ttau \in \LL^2(\Omega);\; \ttau|_K \in\HH(\div,K) \quad\forall K \in\cT_n\}
\end{align*}
and trace spaces
\begin{align*}
   H_{00}^{1/2}(\cS_n)
   &:=
   \Bigl\{\hat w\in\prod_{K\in\cT_n} H^{1/2}(\partial K);\;
     \exists w\in H^1_0(\Omega):\
     \hat w|_{\partial K} = w|_{\partial K}\ \forall K\in\cT_n\Bigr\},\\
   H^{-1/2}(\cS_n)
   &:=
   \Bigl\{\hat q\in\prod_{K\in\cT_n} H^{-1/2}(\partial K);\;
     \exists\qq\in\HH(\div,\Omega):\
     \hat q|_{\partial K} = (\qq\cdot\nn_K)|_{\partial K}\ \forall K\in\cT_n\Bigr\}.
\end{align*}
Here, $\nn_K$ denotes the unit outward normal vector on $\partial K$ ($K\in\cT_n$).
Norms for $H^1(\cT_n)$ and $\HH(\div,\cT_n)$ are defined later,
and the trace spaces are equipped with the weighted trace norms
\begin{subequations} \label{tracenorms}
\begin{align}
   \|\hat w\|_{1/2,\cS_n}
   &:=
   \inf\{(\|w\|^2 + k_n\|\nabla w\|^2)^{1/2};\;
         w\in H_0^1(\Omega),\ \hat w|_{\partial K} = w|_{\partial K}\ \forall K\in\cT_n\},
   \label{tracenorm_p}
   \\
   \|\hat q\|_{-1/2,\cS_n}
   &:=
   \inf\{ (\|\qq\|^2 + k_n \|\div\qq\|^2)^{1/2};\; \qq\in \HH(\div,\Omega),\
                       \hat q|_{\partial K} = (\qq\cdot\nn_K)|_{\partial K}\ \forall K\in\cT_n\}.
   \label{tracenorm_m}
\end{align}
\end{subequations}
Our variational formulation will give rise to the terms
\begin{align*}
   \<\jump{\ttau\cdot\nn},\hat w\>_{\cS_n}
   &:=
   \sum_{K\in\cT_n} \<\ttau\cdot\nn_K,\hat w\>_{\partial K}
   &&(\ttau\in \HH(\div,\cT_n),\ \hat w\in H_{00}^{1/2}(\cS_n)),\\
   \<\jump{v},\hat q\>_{\cS_n}
   &:=
   \sum_{K\in\cT_n} \<v,\hat q\>_{\partial K}
   &&(v\in H^1(\cT_n),\ \hat q\in H^{-1/2}(\cS_n)).
\end{align*}
Here, we will not distinguish between $\<\jump{v},\hat q\>_{\cS_n}$ and $\<\hat q,\jump{v}\>_{\cS_n}$,
and similarly for $\jump{\ttau\cdot\nn}$.
Formally, $\jump{\ttau\cdot\nn}$ and $\jump{v}$ are linear functionals acting on
$ H_{00}^{1/2}(\cS_n)$ and $H^{-1/2}(\cS_n)$, respectively, and are measured accordingly:
\begin{align*}
   \|\jump{\ttau\cdot\nn}\|_{-1/2,\cS_n'}
   &:=
   \sup_{\hat w\in H_{00}^{1/2}(\cS_n)} \frac{\<\jump{\ttau\cdot\nn}, \hat w\>_{\cS_n}}
                                             {\|\hat w\|_{1/2,\cS_n}},\qquad
   \|\jump{v}\|_{1/2,\cS_n'}
   :=
   \sup_{\hat q\in H^{-1/2}(\cS_n)} \frac{\<\jump{v}, \hat q\>_{\cS_n}}
                                         {\|\hat q\|_{-1/2,\cS_n}}.
\end{align*}
Here and in the following, suprema will always be taken over nonzero elements.

Indeed, $\jump{v}$ and $\jump{\ttau\cdot\nn}$ are related to the jumps of $v$ and $\ttau$,
and $\jump{v}=0$, $\jump{\ttau\cdot\nn}=0$ for $v\in H^1_0(\Omega)$ and $\ttau\in\HH(\div,\Omega)$,
respectively. Note that $\jump{\ttau\cdot\nn}=0$ does not require that $\ttau$ has normal
component zero on $\Gamma$ since, in our definition, it is only tested with functions that have
zero trace on $\Gamma$.

\subsection{Semi-discrete scheme} \label{sec_22}

DPG analysis is based on studying variational formulations. The analysis of discrete DPG schemes
is inherited from the continuous case. Since we aim at time-stepping schemes we first apply
a time discretization to \eqref{heat} and analyze the resulting scheme in a space-variational form.
For simplicity we consider a backward Euler discretization. Using the previously defined
time steps $0=t_0 < t_1 < \ldots < t_N=T$, we obtain the semi-discrete approximation
\begin{subequations} \label{semi}
\begin{alignat}{2}
    \frac 1{k_n}u^n - \Delta u^n &= f^n+\frac 1{k_n}u^{n-1}
    &\qquad& \text{in}\quad\Omega, \label{semia} \\
    u^n &= 0    && \text{on}\quad\Gamma, \label{semib}
\end{alignat}
\end{subequations}
for $n=1,\ldots,N$, with initial condition $u^0=u_0$ in $\Omega$.
It defines approximations $u^n$ to $u(t_n):=u(\cdot,t=t_n)$, $n=1,\ldots,N$.
Here, $f^n:=f(t_n)$ where, for space-time functions $v$,
we generically denote by $v(t)$ the space function $v$ at time $t$.

Of course, for $f(t)\in L^2(\Omega)$ ($t\in (0,T]$), \eqref{semi} is uniquely solvable
with $u^n\in H^1_0(\Omega)$, $n=1,\ldots,N$.
In principle, we can rewrite \eqref{semi} in any variational form that renders the Laplacian
well posed. The strong form is \eqref{semi} and one can equivalently use the standard formulation
(based on the Dirichlet bilinear form), an ultra-weak formulation, or any variant in between,
set in global spaces on $\Omega$ or broken variants set in spaces on $\cT_n$.
For details we refer to \cite{CarstensenDG_BSF}.

Rewriting the Laplacian in \eqref{semia} in variational form, one has to add
the term $u^n/k_n$ which makes the formulation formally singularly perturbed. In other words,
one naturally tends to analyze the resulting formulation with particular focus on the parameter $k_n$.
However, a well-posed variational form of \eqref{semi} will be equivalent to \eqref{semi}
(with \eqref{semia} taken in $H^{-1}(\Omega)$) and, thus, automatically stable
(cf.~Proposition~\ref{prop_stab} below).

In this paper we focus on an ultra-weak formulation.
Our aim is to derive a fully discrete approximation based on this ultra-weak form,
and to provide a continuous stability analysis that implies stability of the fully discrete approximation.

We take the ultra-weak variational formulation from \cite{DemkowiczG_11_ADM}. More precisely,
we introduce $\ssigma^n:=\nabla u^n$ as further unknowns and define the independent trace
variables $\hat u^n$, $\hat\sigma^n$ with $\hat u^n|_{\partial K}=u^n|_{\partial K}$,
respectively $\hat\sigma^n|_{\partial K}=(\ssigma^n\cdot\nn_K)|_{\partial K}$,
$K\in\cT_n$, $n=1,\ldots,N$. Replacing $\nabla u^n=\ssigma^n$ in \eqref{semia} and testing with
$v\in H^1(\cT_n)$, testing $\nabla u^n=\ssigma^n$ with $\ttau\in\HH(\div,\cT_n)$, and
integrating by parts, we obtain
\begin{align*}
   \frac 1{k_n}(u^n,v) + (\ssigma^n,\nabla v)_{\cT_n} - \<\hat\sigma^n,\jump{v}\>_{\cS_n}
   &=
   (f^n+\frac 1{k_n}u^{n-1},v),\\
   (\ssigma^n,\ttau) + (u^n,\div\ttau)_{\cT_n} - \<\hat u^n,\jump{\ttau\cdot\nn}\>_{\cS_n}
   &=
   0,
\end{align*}
with $(\cdot,\cdot)_{\cT_n}:=\sum_{K\in\cT_n} (\cdot,\cdot)_K$.
In the following we will denote unknown functions by
$\uu^n=(u^n,\ssigma^n,\hat u^n,\hat\sigma^n)$ (with or without upper index $n$), and test functions
by $\vv=(v,\ttau)$. Note that the bilinear form
\[
   b(\uu,\vv) :=  (u,\div\ttau)_{\cT_n} + (\ssigma,\nabla v+\ttau)_{\cT_n} 
                 - \<\hat u,\jump{\ttau\cdot\nn}\>_{\cS_n} - \<\hat\sigma,\jump{v}\>_{\cS_n}
\]
is nothing but the ultra-weak bilinear form of the Laplacian from \cite{DemkowiczG_11_ADM}.
We obtain the following ultra-weak variational form of \eqref{semi},
\begin{align} \label{semi_weak}
   &\uu^n\in X^n:=L^2(\Omega)\times\LL^2(\Omega)\times H^{1/2}_{00}(\cS_n)\times H^{-1/2}(\cS_n):
   \nonumber\\
   &b_e^n(\uu^n,\vv) = L_e^n(u^{n-1};\vv)
   \qquad\forall\vv\in Y^n:=H^1(\cT_n)\times \HH(\div,\cT_n),\ n=1,\ldots, N,\\
   &\text{with}\quad u^0=u_0,\nonumber
\end{align}
and where
\[
   b_e^n(\uu,\vv):= \frac 1{k_n}(u,v)+b(\uu,\vv)\quad\text{and}\quad
   L_e^n(u;\vv) := (f^n+\frac 1{k_n}u,v).
\]
In some sense, the bilinear and linear forms $b_e^n$ and $L_e^n(u,\cdot)$ extend the corresponding
forms of the Poisson equation. Instead of the Poisson equation, our variational forms \eqref{semi_weak}
represent singularly perturbed reaction-diffusion problems at $N$ time steps.
Note, however, that the perturbations are singular only in a mild sense since, multiplying by $k_n$,
the resulting small diffusion parameter also appears as a factor on the right-hand side.

For completeness let us recall the stability of our semi-discrete scheme.

\begin{prop} \label{prop_stab}
For $u_0,f(t)\in L^2(\Omega)$ ($t\in(0,T]$) the semi-discrete ultra-weak formulation
\eqref{semi_weak} is uniquely solvable, and there holds
(with $\uu^n=(u^n,\ssigma^n,\hat u^n,\hat\sigma^n)$, $n=1,\ldots,N$, denoting the solution)
\[
   \Bigl(\|u^n\|^2 + k_n\|\ssigma^n\|^2\Bigr)^{1/2}
   \le
   \|u^{n-1}\| + k_n\|f^n\|
   \quad (n=1,\ldots,N),
\]
that is,
\[
   \Bigl(\|u^N\|^2 + k_N\|\ssigma^N\|^2\Bigr)^{1/2}
   \le
   \|u_0\| + \sum_{n=1}^N k_n\|f^n\|.
\]
\end{prop}

\begin{proof}
By standard arguments, \eqref{semi_weak} is equivalent to
\[
   \frac 1{k_n}u^n - \Delta u^n = f^n+\frac 1{k_n}u^{n-1}
   \quad\text{in}\quad H^{-1}(\Omega),\quad n=1,\ldots,N,
\]
that is, making use of $\ssigma^n=\nabla u^n$ ($n=1,\ldots,N$),
\[
   (u^n,u^n) + k_n(\ssigma^n,\ssigma^n) = k_n (f^n,u^n) + (u^{n-1},u^n)
   \le
   \Bigl(k_n\|f^n\| + \|u^{n-1}\|\Bigr)\,\|u^n\|.
\]
This proves the first assertion. The second one follows by repeated application of this bound.
\end{proof}

A control of the trace variables is guaranteed by a simple application
of the definition of trace norms.

\begin{cor} \label{cor_stab}
Under the assumptions of Proposition~\ref{prop_stab} there holds
\begin{align*}
   &\|\hat u^n\|_{1/2,\cS_n}
   \le \Bigl(\|u^n\|^2 + k_n \|\ssigma^n\|^2\Bigr)^{1/2}
   \le \|u^{n-1}\| + k_n\|f^n\|,\\
   &k_n^{1/2}\|\hat\sigma^n\|_{-1/2,\cS_n}
   \le (1+\sqrt{2}) \Bigl(\|u^{n-1}\| + k_n\|f^n\|\Bigr)
\end{align*}
for $n=1,\ldots,N$, and, in particular,
\begin{align*}
   &\|\hat u^N\|_{1/2,\cS_N}
   \le \|u_0\| + \sum_{n=1}^N k_n\|f^n\|,\\
   &k_N^{1/2}\|\hat\sigma^N\|_{-1/2,\cS_N}
   \le (1+\sqrt{2}) \Bigl(\|u_0\| + \sum_{n=1}^N k_n\|f^n\|\Bigr).
\end{align*}
\end{cor}

\begin{proof}
By the definition of the trace norms in $H^{1/2}_{00}(\cS_n)$ and $H^{-1/2}(\cS_n)$, and
since the solutions $\uu^n$ of \eqref{semi_weak} satisfy
\[
   \hat u^n|_{\partial K} = u^n|_{\partial K},\quad
   \hat\sigma^n|_{\partial K} = (\ssigma^n\cdot\nn_K)|_{\partial K}
   \quad (K\in\cT_n),
\]
the first statement follows immediately, and for $\hat\sigma^n$ we obtain
\begin{align*}
   k_n \|\hat\sigma^n\|_{-1/2,\cS}^2
   &\le
   k_n \|\ssigma^n\|^2 + k_n^2 \|\div\ssigma^n\|^2.
\end{align*}
Now, $\div\ssigma^n=\Delta u^n=\frac 1{k_n}(u^n-u^{n-1}) - f^n$, so that
\[
   k_n^2 \|\div\ssigma^n\|^2 \le \Bigl(\|u^n\| + \|u^{n-1}\| + k_n\|f^n\|\Bigr)^2,
\]
i.e., with the previous bound,
\begin{align*}
   k_n^{1/2} \|\hat\sigma^n\|_{-1/2,\cS}
   &\le
   k_n^{1/2}\|\ssigma^n\| + k_n\|\div\ssigma^n\|
   \le
   k_n^{1/2}\|\ssigma^n\|
   + \|u^n\| + \|u^{n-1}\| + k_n\|f^n\|\\
   &\le
   \sqrt{2}\Bigl(\|u^n\|^2 + k_n\|\ssigma^n\|^2\Bigr)^{1/2} + \|u^{n-1}\| + k_n\|f^n\|.
\end{align*}
Using the first statement, this implies the second bound. The remaining two assertions
are obtained by iterated applications of the first bound.
\end{proof}

\subsection{Fully discrete scheme} \label{sec_23}

Our DPG approximation will be based on the test norm(s)
\begin{align} \label{testnorm}
   &\|\vv\|_{Y^n}^2
   :=
   \frac 1{k_n^2} \|v\|^2 + \frac 1{k_n} \|\nabla v\|_{\cT_n}^2
   + \frac 1{k_n} \|\ttau\|^2 + \|\div\ttau\|_{\cT_n}^2\\
   &\text{with}\quad \|\nabla v\|_{\cT_n}^2 := (\nabla v,\nabla v)_{\cT_n}\quad
    \text{and}\quad \|\div\ttau\|_{\cT_n}^2 := (\div\ttau,\div\ttau)_{\cT_n}.
   \nonumber
\end{align}
The corresponding inner product will be denoted by $\llangle\cdot,\cdot\rrangle_{Y^n}$.

Now, for a function $\uu\in X^n$, we define the \emph{optimal test function}
\begin{equation} \label{ttt}
  \Theta^n\uu := (\Theta^n_v\uu,\Theta^n_\tau\uu)\in Y^n\quad\text{by}\quad
   \llangle\Theta^n\uu,\vv\rrangle_{Y^n} = b_e^n(\uu, \vv) \quad\forall\vv\in Y^n.
\end{equation}
Then, selecting discrete subspaces (piecewise polynomial with respect to $\cT_n$ and $\cS_n$)
$X^n_h \subset X^n$ ($n=1,\ldots,N$) and, slightly abusing notation,
a discrete space $X^0_h\subset L^2(\Omega)$, the fully discrete scheme is:
\emph{Find $u^0_h\in X^0_h$ and $\uu^n_h=(u^n_h,\ssigma^n_h,\hat u^n_h,\hat\sigma^n_h)\in X^n_h$
($n=1,\ldots,N$) such that}
\begin{subequations} \label{DPG}
\begin{alignat}{2}
   &b_e^n(\uu_h^n, \vv) = L_e^n(u^{n-1}_h;\vv)
   &\quad& \forall\vv\in Y^n_h:=\Theta^n X^n_h\quad (n=1,\ldots,N) \label{DPGa}\\
   &\text{\emph{with}}\quad (u^0_h,v) = (u_0,v) &&\forall v\in X^0_h. \label{DPGb}
\end{alignat}
\end{subequations}
One of our main results is that the discrete scheme inherits its stability from that of the
semi-discrete variational formulation, cf.~Proposition~\ref{prop_stab}.

\begin{theorem} \label{thm_stab}
For $u_0,f(t)\in L^2(\Omega)$ ($t\in(0,T]$) the DPG scheme \eqref{DPG}
is uniquely solvable and, being
$\uu^n_h=(u^n_h,\ssigma^n_h,\hat u^n_h,\hat\sigma^n_h)$, $n=1,\ldots,N$, its solutions,
there holds
\[
   \Bigl(\|u^n_h\|^2 + k_n\|\ssigma^n_h\|^2\Bigr)^{1/2}
   \le
   \|u^{n-1}_h\| + k_n\|f^n\|
   \quad (n=1,\ldots,N),
\]
that is,
\[
   \Bigl(\|u^N_h\|^2 + k_N\|\ssigma^N_h\|^2\Bigr)^{1/2}
   \le
   \|u_0\| + \sum_{n=1}^N k_n\|f^n\|.
\]
\end{theorem}

A proof of this theorem will be given in Section~\ref{sec_pf_stab}.
It will be based on proving norm equivalences in $X^n$ and $Y^n$.
Specifically, we use the following norms in $X^n$ ($n=1,\ldots,N$),
\begin{align*}
   \|\uu\|_{X^n_1}^2
   &:=
   \|u\|^2 + k_n \|\ssigma\|^2 + k_n \|\hat u\|^2_{1/2,\cS_n} + k_n^2 \|\hat\sigma\|^2_{-1/2,\cS_n},\\
   \|\uu\|_{X^n_2}^2
   &:=
   \|u\|^2 + k_n \|\ssigma\|^2 + \|\hat u\|^2_{1/2,\cS_n} + k_n \|\hat\sigma\|^2_{-1/2,\cS_n},
   \quad \uu \in X^n.
\end{align*}
However, central to DPG analysis is the \emph{energy norm}. In our time-stepping scheme we have
a family of energy norms
\begin{equation} \label{energy}
   \|\uu\|_{E^n} := \sup_{\vv\in Y^n} \frac{b_e^n(\uu,\vv)}{\|\vv\|_{Y^n}},
   \quad \uu \in X^n,\ n=1,\ldots,N.
\end{equation}
By the well-posedness of problem \eqref{semi_weak}, these are indeed norms,
i.e., $\uu\in X^n$ with $b_e^n(\uu,\vv)=0$ for any $\vv\in Y^n$ implies that $\uu=0$.

With this notation we can state the second main result.
It shows quasi-optimality of our time-stepping DPG scheme.

\begin{theorem} \label{thm_Cea}
Let $u_0,f(t)\in L^2(\Omega)$ for $t\in(0,T]$, and let
$u$ and $\uu^n_h\in X^n_h$ ($n=1,\ldots,N$), $u^0_h$ denote the solutions of
\eqref{heat} and \eqref{DPG}, respectively. Furthermore, we denote
$\uu:=(u,\nabla u,\hat u,\hat\sigma)$ with $\hat u$, $\hat\sigma$ being the traces of
$u$ and $\nabla u$, respectively, and let $u^n_h$ be the first component of $\uu^n_h$.
Then there holds the following error estimate for $n=1,\ldots,N$:
\begin{align*}
   \|\uu(t_n)-\uu^n_h\|_{E^n}
   &\le
   \|u_0-u^0_h\|
   +
   \sum_{j=1}^n \min_{\ww\in X^j_h} \|\uu(t_j)-\ww\|_{E^j}
   +
   \sum_{j=1}^n \|u(t_n)-u(t_{n-1}) - k_n\dot u(t_n)\|.
\end{align*}
Furthermore, $\|\uu(t_n)-\uu^n_h\|_{X^n_1}$ has the bound above multiplied by
$\sqrt{2}\max\Bigl\{1,\sqrt{4C_\mathrm{PF}^2+6k_n}\Bigr\}$.
Here, $C_\mathrm{PF}$ denotes the Poincar\'e-Friedrichs constant,
i.e., $\|w\|\le C_\mathrm{PF}\|\nabla w\|$ $\forall w\in H^1_0(\Omega)$.

The best approximations in energy norm can be estimated like
\[
   \min_{\ww\in X^n_h} \|\uu(t_n)-\ww\|_{E^n}
   \le
   \sqrt{3} \min_{\ww\in X^n_h} \|\uu(t_n)-\ww\|_{X^n_2},\qquad n=1,\ldots,N.
\]
\end{theorem}

A proof of this result will be given in Section~\ref{sec_pf_Cea}.
Applying standard approximation results one obtains convergence estimates.
In the case of smooth solution, quasi-uniform meshes in space (fixed in time),
and equal time steps, one obtains the following result.

\begin{cor} \label{cor_conv}
Assume that $u_0$ and $f$ are sufficiently smooth so that the solution\\
$u\in L^\infty(0,T;H^3(\Omega))\cap W^{2,1}(0,T;L^2(\Omega))$.
Furthermore, let $\cT_n=\cT$ ($n=0,\ldots,N$) be quasi-uniform triangular meshes
with skeleton $\cS$ and mesh size $h$, and let $k_n=k$ ($n=1,\ldots,N$) be constant time steps.

(i) Taking lowest order spaces $X^n_h$ ($\cT$-piecewise constants for $u$ and $\ssigma$,
$\cS$-piecewise linears for $\hat u$, and $\cS$-piecewise constants for $\hat\sigma$),
there holds
\[
   \|u(T)-u^N_h\| + k^{1/2}\|\nabla u(T)-\ssigma^N_h\|
   =
   O(h/k) + O(k),
\]
that is, e.g., for $k=O(h^{1/2})$,
\[
   \|u(T)-u^N_h\|=O(h^{1/2})=O(k), \qquad \|\nabla u(T)-\ssigma^N_h\|=O(h^{1/4})=O(k^{1/2}).
\]
(ii) Taking lowest order spaces for the $\ssigma$, $\hat u$ and $\hat\sigma$-components of $X^n_h$,
and piecewise linears for the $u$ component of $X^n_h$, there holds
\[
   \|u(T)-u^N_h\| + k^{1/2}\|\nabla u(T)-\ssigma^N_h\|
   = O(h/k^{1/2}) + O(k)
   \quad\text{if}\quad h\le O(k^{1/2}),
\]
that is, e.g., for $k=O(h^{2/3})$,
\[
   \|u(T)-u^N_h\|=O(h^{2/3})=O(k),\qquad \|\nabla u(T)-\ssigma^N_h\|=O(h^{1/3})=O(k^{1/2}).
\]
\end{cor}

\begin{proof}
(i) There holds $\|u_0-u^0_h\|=O(h)$, $\|u(t_n)-u(t_{n-1}) - k_n\dot u(t_n)\|=O(k^2)$ so that
the first result follows from Theorem~\ref{thm_Cea} by showing that, with
$\ww=(w^u,\ww^\sigma,\hat w^u,\hat w^\sigma)$,
\begin{align*}
\lefteqn{
   \min_{\ww\in X^n_h} \|\uu(t_n)-\ww\|_{X^n_2}^2
   =
}\\
&
   \min_{\ww\in X^n_h}
   \Bigl(
      \|u(t_n)-w^u\|^2 + k \|\ssigma(t_n) - \ww^\sigma\|^2
    + \|\hat u - \hat w^u\|_{1/2,\cS}^2 + k \|\hat\sigma-\hat w^\sigma\|_{-1/2,\cS}^2
   \Bigr)
   =O(h^2).
\end{align*}
Indeed, the first two terms are of the orders $O(h^2)$ and $O(kh^2)$.
By the definition of the norm $\|\cdot\|_{1/2,\cS}$ as the trace of the square root of
$\|\cdot\|^2+k\|\nabla\cdot\|^2$, the third term is of the order $O(h^4)+O(kh^2) = O(h^2)$.
The norm $\|\cdot\|_{-1/2,\cS}$ is defined as the trace of the square root of
$\|\cdot\|^2+k\|\div\cdot\|^2$. Thus, the fourth term is of the order $O(kh^2)$.\\
(ii) From the first part we have seen that the $\hat u$-term, respectively
$\ssigma$ and $\hat\sigma$-terms, of the error (squared) are of the orders
$O(h^4)+O(kh^2)$ and $O(kh^2)$, respectively.
The $u$-term (squared) is now of the order $O(h^4)$. It follows that
\[
   \min_{\ww\in X^n_h} \|\uu(t_n)-\ww\|_{X^n_2} = O(h^2) + O(k^{1/2}h),
\]
and thus the result.
\end{proof}

\section{Analysis} \label{sec_anal} \label{sec_3}

In this section we analyze our time-stepping DPG scheme, and finish with proving
the main results, Theorem~\ref{thm_stab} in \S\ref{sec_pf_stab} and
Theorem~\ref{thm_Cea} in \S\ref{sec_pf_Cea}. As a preparation, we show some
stability properties of adjoint problems in \S\ref{sec_adj} and, in \S\ref{sec_norms},
prove a pivotal norm equivalence in the ansatz spaces.

Let us start by recalling some basic facts of DPG analysis, see, e.g., \cite{DemkowiczG_11_CDP}.

As already mentioned, central to DPG analysis is the energy norm.
In our time-stepping case there is one for every time step $t_n$ ($n=1,\ldots,N$),
$\|\cdot\|_{E^n}$, see~\eqref{energy}.
Denoting by $B^n_e:\;X^n\to (Y^n)'$ the operator stemming from the bilinear form $b_e^n$, and
$\cR^n:\;Y^n\to (Y^n)'$ the Riesz operator, one finds the representation
$\Theta^n=(\cR^n)^{-1}B_e^n$ for the trial-to-test operator, cf.~\eqref{ttt}. Therefore,
\begin{equation} \label{energy_ttt}
   \|\uu\|_{E^n} = \|B_e^n\uu\|_{(Y^n)'} = \|\Theta^n\uu\|_{Y^n}
   \qquad\forall\uu\in X^n.
\end{equation}
We need the following useful relations for discrete functions which hold for
our choice of test spaces $Y^n_h=\Theta^n X^n_h$:
\begin{equation} \label{Enorm_discrete}
   \|\uu\|_{E^n} = b_e^n(\uu,\Theta^n\uu)^{1/2}
                 = \sup_{\vv\in Y^n_h} \frac {b_e^n(\uu,\vv)}{\|\vv\|_{Y^n}}
   \qquad\forall \uu\in X^n_h.
\end{equation}
The first equality is obtained by using \eqref{energy_ttt} and relation \eqref{ttt},
\[
   \|\uu\|_{E^n} = \llangle\Theta^n\uu,\Theta^n\uu\rrangle_{Y^n}^{1/2} = b_e^n(\uu,\Theta^n\uu)^{1/2}.
\]
The second equality in \eqref{Enorm_discrete} is obtained by construction of
$Y^n_h$, containing the optimal test function $\Theta^n\uu$ for $\uu\in X^n_h$.

\subsection{Stability of the adjoint problem} \label{sec_adj}

Corresponding to each ultra-weak formulation at a given time step $t_n$, cf.~\eqref{semi_weak},
there is a primal problem (the reaction-diffusion problems \eqref{semi}) and
an adjoint problem. Therefore, our time-stepping scheme gives rise to $N$ adjoint problems
whose stability is essential for the stability and robustness of the DPG scheme, as we will
illustrate when proving our main results. As is standard in DPG analysis, we split the
adjoint problem(s) in a global one (with solution in continuous spaces) and a homogeneous
one (with solution in broken spaces). The corresponding results are established by
Lemmas~\ref{la_stab} and \ref{la_stab_hom}, respectively.

\begin{lemma} \label{la_stab}
Let $G\in L^2(\Omega)$ and $\FF\in\LL^2(\Omega)$ be given.
There exists a unique element $\vv=(v,\ttau)\in H^1_0(\Omega)\times\HH(\div,\Omega)$ such that
\begin{subequations} \label{stab}
\begin{alignat}{2}
   \frac 1{k_n}v + \div\ttau &= G   &\qquad& \text{in}\quad\Omega, \label{staba}\\
   \nabla v + \ttau          &= \FF &&       \text{in}\quad\Omega, \label{stabb}
\end{alignat}
\end{subequations}
and there holds
\[
   \|\vv\|_{Y^n}^2 = \|G\|^2 + \frac 1{k_n} \|\FF\|^2.
\]
\end{lemma}

\begin{proof}
Eliminating $\ttau$ from \eqref{stab} gives
\begin{align} \label{pf_stab_1}
   -\Delta v + \frac 1{k_n}v = G - \div\FF \qquad\text{in}\quad  H^{-1}(\Omega),
\end{align}
with unique solution $v\in H^1_0(\Omega)$.
Equation~\eqref{stabb} then uniquely defines $\ttau\in\LL^2(\Omega)$, and by \eqref{pf_stab_1},
$\div\ttau=\div\FF-\Delta v=G-k_n^{-1}v$, that is, $\ttau\in\HH(\div,\Omega)$
and it satisfies \eqref{staba}.

Now, using both equations \eqref{staba} and \eqref{stabb},
\begin{align*}
\lefteqn{
   \|G\|^2 + \frac 1{k_n} \|\FF\|^2
   =
   \|k_n^{-1}v+\div\ttau\|^2 + \frac 1{k_n} \|\nabla v+\ttau\|^2
}\\
   &=
   \frac 1{k_n^2} \|v\|^2 + \|\div\ttau\|^2 + \frac 2{k_n} (v,\div\ttau)
   +
   \frac 1{k_n} \|\nabla v\|^2 + \frac 1{k_n}\|\ttau\|^2 + \frac 2{k_n} (\nabla v,\ttau)
   =
   \|\vv\|_{Y^n}^2
\end{align*}
since $(v,\div\ttau)+(\nabla v,\ttau)=0$.
\end{proof}

\begin{lemma}\label{la_stab_hom}
If $\vv=(v,\ttau)\in Y^n=H^1(\cT_n)\times\HH(\div, \cT_n)$ satisfies
\begin{subequations} \label{stabh}
\begin{alignat}{2}
   \frac 1{k_n} v  + \div\ttau &= 0  &\qquad& \text{in}\quad K, \label{stabha}\\
   \nabla v + \ttau            &= 0  &&       \text{in}\quad K, \label{stabhb}
\end{alignat}
\end{subequations}
for any $K\in\cT_n$, then
\begin{align}
   \frac 1{k_n^2} \|v\|^2 = \|\div\ttau\|_{\cT_n}^2
   &\le
   \|\jump{\ttau\cdot\nn}\|_{-1/2,\cS_n'}^2
   + \frac 1{k_n} \|\jump{v}\|_{1/2,\cS_n'}^2, \label{stabhc}
   \\
   \|\nabla v\|_{\cT_n}^2 = \|\ttau\|^2
   &\le
   4(C_\mathrm{PF}^2+k_n) \Bigl( \|\jump{\ttau\cdot\nn}\|_{-1/2,\cS_n'}^2
   + \frac 1{k_n} \|\jump{v}\|_{1/2,\cS_n'}^2 \Bigr) \label{stabhd}
\end{align}
that is,
\begin{align*}
   \|\vv\|_{Y^n}^2
   \le
   \frac1{k_n} (4C_\mathrm{PF}^2 + 6 k_n)
   \Bigl( \|\jump{\ttau\cdot\nn}\|_{-1/2,\cS_n'}^2
        + \frac 1{k_n} \|\jump{v}\|_{1/2,\cS_n'}^2 \Bigr)
\end{align*}
for $n=1,\ldots,N$.
\end{lemma}

\begin{proof}
To bound $\|v\|$ we use the technique from \cite[proof of Lemma 8]{HeuerK_RDM},
but are careful with constants.
We define $w\in H^1_0(\Omega)$ as the solution of $-k_n\Delta w + w = v$. Then
\(
   \|w\|^2 + k_n \|\nabla w\|^2  = (v,w)
\)
and, by our definition of trace norms \eqref{tracenorms},
$\|w\|_{1/2,\cS_n}^2 \le (v,w)$ and
\begin{align*}
   k_n \|\nn\cdot\nabla w\|_{-1/2,\cS_n}^2
   &\le
   k_n \|\nabla w\|^2 + k_n^2 \|\Delta w\|^2\\
   &=
   k_n \|\nabla w\|^2 + \|w\|^2 + \|v\|^2 - 2(v,w)
   =
   \|v\|^2 - (v,w).
\end{align*}
Therefore,
\[
   \|w\|_{1/2,\cS_n}^2 + k_n \|\nn\cdot\nabla w\|_{-1/2,\cS_n}^2 \le \|v\|^2.
\]
Now, using the equation for $w$, integrating twice piecewise by parts, and
considering the bound above for the trace norms of $w$, we obtain
\begin{align*}
   \frac 1{k_n} \|v\|^2
   =
   \frac 1{k_n} (v,w) - (v,\Delta w)
   &=
   (\frac 1{k_n}v - \Delta v, w)_{\cT_n}
   + \<\jump{\nn\cdot\nabla v},w\>_{\cS_n} - \<\jump{v},\nn\cdot\nabla w\>_{\cS_n}\\
   &=
   \<\jump{\nn\cdot\nabla v},w\>_{\cS_n} - \<\jump{v}, \nn\cdot\nabla w\>_{\cS_n}\\
   &\le
   \Bigl(\|\jump{\nn\cdot\nabla v}\|_{-1/2,\cS_n'}^2
         +
         \frac 1{k_n}\|\jump{v}\|_{1/2,\cS_n'}^2
   \Bigr)^{1/2}
   \|v\|.
\end{align*}
Since $\jump{\nn\cdot\nabla v}=-\jump{\nn\cdot\ttau}$ by \eqref{stabhb}, this yields
\[
   \frac 1{k_n^2} \|v\|^2
   \le
   \|\jump{\ttau\cdot\nn}\|_{-1/2,\cS_n'}^2 + \frac 1{k_n}\|\jump{v}\|_{1/2,\cS_n'}^2.
\]
By \eqref{stabha}, this proves \eqref{stabhc}.

Now, to show \eqref{stabhd}, we follow \cite[proof of Lemma 4.4]{DemkowiczG_11_ADM}
by considering, in three dimensions, the Helmholtz decomposition
\(
   \ttau = \nabla\psi + \nabla\times\zz
\)
with $\psi \in H^1_0(\Omega)$ and $\zz\in\HH(\curl,\Omega)$, so that
$\|\nabla\psi\| \le \|\ttau\|$ and $\|\nabla\times\zz\| \le \|\ttau\|$.
Then, integrating by parts piecewise, one obtains
\begin{align*}
   \|\ttau\|^2
   &=
   (\ttau, \nabla\psi + \nabla\times\zz)
   =
   (\ttau, \nabla\psi) - (\nabla v,\nabla\times\zz)_{\cT_n}\\
   &=
   \frac 1{k_n} (v,\psi)
   + \<\jump{\ttau\cdot\nn},\psi\>_{\cS_n}
   - \<\jump{v}, \nn\cdot(\nabla\times\zz)\>_{\cS_n}\\
   &\le
   \frac {C_\mathrm{PF}}{k_n} \|v\| \| \ttau\|
   + \|\jump{\ttau\cdot\nn}\|_{-1/2,\cS_n'} \|\psi\|_{1/2,\cS_n}
   + \|\jump{v}\|_{1/2,\cS_n'}  \|\nn\cdot(\nabla\times\zz)\|_{-1/2,\cS_n}.
\end{align*}
By the definition of trace norms we have
\begin{align*}
   &\|\psi\|_{1/2,\cS_n}^2
   \le
   \|\psi\|^2 + k_n \|\nabla\psi\|^2
   \le
   (C_\mathrm{PF}^2+k_n) \|\ttau\|^2
   \qquad\text{and}\\
   &\|\nn\cdot(\nabla\times\zz)\|_{-1/2,\cS_n}
   \le
   \|\nabla\times\zz\| \le \|\ttau\|.
\end{align*}
Therefore, together with \eqref{stabhc}, we conclude that
\begin{align*}
  \|\ttau\| &\leq C_\mathrm{PF} \Bigl( \|\jump{\ttau\cdot\nn}\|_{-1/2,\cS_n'}^2 + \frac 1{k_n}
  \|\jump{v}\|_{1/2,\cS_n'}^2 \Bigr)^{1/2} \\
  &\qquad + \Bigl(C_\mathrm{PF}^2 + k_n\Bigr)^{1/2} \|\jump{\ttau\cdot\nn}\|_{-1/2,\cS_n'} + 
  k_n^{1/2} \frac1{k_n^{1/2}} \|\jump{v}\|_{1/2,\cS_n'} \\
  &\leq C_\mathrm{PF} \Bigl( \|\jump{\ttau\cdot\nn}\|_{-1/2,\cS_n'}^2 + \frac 1{k_n} 
  \|\jump{v}\|_{1/2,\cS_n'}^2 \Bigr)^{1/2} \\
  &\qquad + \Bigl( C_\mathrm{PF}^2 + 2k_n\Bigr)^{1/2} 
  \Bigl( \|\jump{\ttau\cdot\nn}\|_{-1/2,\cS_n'}^2 + \frac 1{k_n} 
  \|\jump{v}\|_{1/2,\cS_n'}^2 \Bigr)^{1/2} \\
  &= \Bigl(C_\mathrm{PF} + \bigl(C_\mathrm{PF}^2+2k_n\bigr)^{1/2} \Bigr)
  \Bigl( \|\jump{\ttau\cdot\nn}\|_{-1/2,\cS_n'}^2 + \frac 1{k_n} 
  \|\jump{v}\|_{1/2,\cS_n'}^2 \Bigr)^{1/2} \\
  &\leq 2\Bigl(C_\mathrm{PF}^2+k_n\Bigr)^{1/2}
  \Bigl( \|\jump{\ttau\cdot\nn}\|_{-1/2,\cS_n'}^2 + \frac 1{k_n} 
  \|\jump{v}\|_{1/2,\cS_n'}^2 \Bigr)^{1/2}.
\end{align*}
This is \eqref{stabhd} for $\ttau$ and the relation for $\nabla v$ holds by \eqref{stabhb}.
In two space dimensions one uses the Helmholtz decomposition
\(
   \ttau = \nabla\psi + (-\partial_2 z, \partial_1 z)
\)
and proceeds as before.
\end{proof}

\subsection{Norm equivalences in $X^n$} \label{sec_norms}

The DPG method delivers best approximations in energy norms.
Estimates in other norms require to bound the energy norms appropriately.
This is provided by Lemma~\ref{la_norms} below.

For its proof we need to introduce so-called \emph{optimal test norms}.
They are dual to $X^n$-norms with respect to the extended bilinear form.
Specifically, we define
\begin{align} \label{testnorm_opt}
   \|\vv\|_{Y^n,\mathrm{opt}}
   &:=
   \sup_{\uu\in X^n} \frac {b_e^n(\uu,\vv)}{\|\uu\|_{X^n_1}},
   \qquad \vv\in Y^n.
\end{align}
By the well-posedness of our semi-discrete formulation \eqref{semi_weak},
the operator $B^n_e:\;X^n\to (Y^n)'$ stemming from the extended bilinear form
is an isomorphism. Thus, there holds
\begin{align} \label{norm_dual}
   \|\uu\|_{X^n_1}
   =
   \sup_{\vv\in Y^n} \frac {b_e^n(\uu,\vv)}{\|\vv\|_{Y^n,\mathrm{opt}}},
   \qquad \uu\in X^n,
\end{align}
see \cite[eq. (2.13)]{ZitelliMDGPC_11_CDP}.

\begin{lemma} \label{la_norms}
There hold the three bounds
\[
   \|\uu\|_{X^n_1} \le C_n \|\uu\|_{E^n},\quad
   \|u\|^2 + k_n\|\ssigma\|^2 \le \|\uu\|_{E^n}^2,\quad\text{and}\quad
   \|\uu\|_{E^n} \le \sqrt{3} \|\uu\|_{X^n_2}
\]
for any $\uu=(u,\ssigma,\hat u,\hat\sigma)\in X^n$, $n=1,\ldots,N$.
Here, $C_n=\sqrt{2}\max\Bigl\{1,\sqrt{4C_\mathrm{PF}^2+6k_n}\Bigr\}$.
\end{lemma}

\begin{proof}
Comparing relation \eqref{norm_dual} and definition \eqref{energy} of the energy norm,
the first statement is equivalent to
\(
   \|\vv\|_{Y^n} \le C \|\vv\|_{Y^n,\mathrm{opt}}
\)
for any $\vv\in Y^n$.
By definition \eqref{testnorm_opt}, one finds that
\[
    \|\vv\|_{Y^n,\mathrm{opt}}^2
    =
    \|\frac 1{k_n}v + \div\ttau\|_{\cT_n}^2 + \frac 1{k_n}\|\nabla v+\ttau\|_{\cT_n}^2
    + \frac 1{k_n} \|\jump{\ttau\cdot\nn}\|_{-1/2,\cS'}^2
    + \frac 1{k_n^2} \|\jump{v}\|_{1/2,\cS'}^2,
\]
for $\vv\in Y^n$.
Now, for given $\vv=(v,\ttau)\in Y^n$, we define $G:=v/k_n + \pwdiv\ttau$, $\FF:=\pwnabla v+\ttau$.
Here, $\pwdiv$ and $\pwnabla$ denote the respective operators defined piecewise on the mesh $\cT_n$.
Then by Lemma~\ref{la_stab} there exists
$\vv_1:=(v_1,\ttau_1)\in H^1_0(\Omega)\times\HH(\div,\Omega)$ with
$\|\vv_1\|_{Y^n}^2=\|G\|^2 + \|\FF\|^2/k_n$. Furthermore, $\vv_0:=(v_0,\ttau_0):=\vv-\vv_1$
solves \eqref{stabh}, and $\jump{v}=\jump{v_0}$, $\jump{\ttau\cdot\nn}=\jump{\ttau_0\cdot\nn}$.
Then the mentioned relation for $\|\vv_1\|_{Y^n}$ and the final bound given
by Lemma~\ref{la_stab_hom} prove that
\begin{align*}
   \frac 12 \|\vv\|_{Y^n}^2
   &\le
   \|\vv_1\|_{Y^n}^2 + \|\vv_0\|_{Y^n}^2
   \\
   &\le
   \|G\|^2 + \frac 1{k_n}\|\FF\|^2
   +
   (4C_\mathrm{PF}^2 + 6 k_n) \Bigl( \frac1{k_n} \|\jump{\ttau\cdot\nn}\|_{-1/2,\cS_n'}^2
   + \frac 1{k_n^2} \|\jump{v}\|_{1/2,\cS_n'}^2 \Bigr)
   \\
   &\le
   \max\{1,4C_\mathrm{PF}^2+6k_n\} \|\vv\|_{Y^n,\mathrm{opt}}^2.
\end{align*}
This proves the first assertion.
By the same reasoning as before, using relation \eqref{norm_dual} and considering
continuous test functions $\vv\in H^1_0(\Omega)\times\HH(\div,\Omega)$, the second
statement is equivalent to the stability result of Lemma~\ref{la_stab}.

To show the third assertion we only have to bound the extended bilinear form
\[
   b_e^n(\uu,\vv)
   =
   (u,\frac 1{k_n} v + \div\ttau)_{\cT_n} + (\ssigma,\nabla v+\ttau)_{\cT_n} 
   - \<\hat u,\jump{\ttau\cdot\nn}\>_{\cS_n} - \<\hat\sigma,\jump{v}\>_{\cS_n}.
\]
This is done by the Cauchy-Schwarz inequality and dualities.
By definition of the skeleton dualities and trace norms we find
\begin{align*}
   \<\hat u,\jump{\ttau\cdot\nn}\>_{\cS_n}^2
   &\le
   \|\hat u\|_{1/2,\cS_n}^2 \Bigl(\frac 1{k_n}\|\ttau\|^2 + \|\div\ttau\|_{\cT_n}^2\Bigr),
   \\
   \<\hat\sigma,\jump{v}\>_{\cS_n}^2
   &\le
   k_n \|\hat\sigma\|_{-1/2,\cS_n}^2
   \Bigl(\frac 1{k_n^2}\|v\|^2 + \frac 1{k_n}\|\nabla v\|_{\cT_n}^2\Bigr),
\end{align*}
so that
\begin{align*}
   b_e^n(\uu,\vv)
   &\le
   \Bigl(\|u\|^2 + k_n\|\ssigma\|^2\Bigr)^{1/2}
   \Bigl(
      \|\frac 1{k_n}v+\div\ttau\|_{\cT_n}^2 + \frac 1{k_n}\|\nabla v+\ttau\|_{\cT_n}^2
   \Bigr)^{1/2}
   \\
   &+
   \Bigl(
      \|\hat u\|_{1/2,\cS_n}^2 + k_n\|\hat\sigma\|_{-1/2,\cS_n}^2
   \Bigr)^{1/2}
   \Bigl(
      \frac 1{k_n^2}\|v\|^2 + \frac 1{k_n}\|\nabla v\|_{\cT_n}^2
      +
      \frac 1{k_n}\|\ttau\|^2 + \|\div\ttau\|_{\cT_n}^2
   \Bigr)^{1/2}
   \\
   &\le
   \sqrt{3} \|\uu\|_{X^n_2} \|\vv\|_{Y^n}\qquad\forall\uu\in X^n,\ \vv\in Y^n,
\end{align*}
that is,
\(
   \|\uu\|_{E^n}\le \sqrt{3} \|\uu\|_{X^n_2}
\)
for any $\uu\in X^n$.
\end{proof}

\subsection{Proof of Theorem~\ref{thm_stab}} \label{sec_pf_stab}

By design of the DPG method (i.e., selecting optimal test functions), \eqref{DPG}
is uniquely solvable (the initial approximation $u^0_h$, being an $L^2$-projection,
exists and is stable anyway). More precisely, the extended bilinear forms
$b_e^n$ are $X^n_h\times Y^n_h$ inf-sup stable and bounded with inf-sup and continuity
numbers equal to $1$, when using the energy norm(s) in $X^n_h$ that correspond(s) to
the selected $Y^n$-norm(s) in $Y^n_h$, cf.~\eqref{testnorm}.

Now, using \eqref{Enorm_discrete}, one obtains by the definition of
$\uu^n_h$ \eqref{DPGa} and the selection of test norm in $Y^n_h$
(remember the notation $\vv=(v,\ttau)$)
\begin{align*}
   \|\uu^n_h\|_{E^n}
   &=
   \sup_{\vv\in Y^n_h} \frac {b_e^n(\uu^n_h,\vv)}{\|\vv\|_{Y^n}}
   =
   \sup_{\vv\in Y^n_h} \frac {L_e^n(u^{n-1}_h,\vv)}{\|\vv\|_{Y^n}}\\
   &\le
   \sup_{\vv\in Y^n_h} \frac {\|f^n\|\, \|v\| + k_n^{-1} \|u^{n-1}_h\| \|v\|}
                           {k_n^{-1} \|v\|}
   =
   k_n \|f^n\| + \|u^{n-1}_h\|.
\end{align*}
Therefore, to finish the proof of Theorem~\ref{thm_stab}, it is enough to show that
\begin{align*}
   &\|u^n_h\|^2 + k_n \|\ssigma^n_h\|^2 \le \|\uu^n_h\|_{E^n}^2,\quad
   n=1,\ldots, N.
\end{align*}
Indeed, there holds the more general result
\begin{align} \label{stab_gen}
   &\|u\|^2 + k_n \|\ssigma\|^2 \le \|\uu\|_{E^n}^2\quad\forall\uu\in X^n,\ n=1,\ldots, N.
\end{align}
By Lemma~\ref{la_stab} we find for any given $G\in L^2(\Omega)$ and $\FF\in\LL(\Omega)$
an element $\vv=\vv(G,\FF)\in H^1_0(\Omega)\times \HH(\div,\Omega)\subset Y^n$ such that
$k_n^{-1}v+\div\ttau=G$ and $\nabla v+\ttau=\FF$ with
$\|\vv\|_{Y^n}^2=\|G\|^2 + k_n^{-1}\|\FF\|^2$.
By this construction we obtain
\begin{align*}
   \Bigl(\|u\|^2 + k_n \|\ssigma\|^2\Bigr)^{1/2}
   =
   \sup_{G\in L^2(\Omega),\, \FF\in\LL^2(\Omega)}
   \frac {(u,G) + (\ssigma,\FF)}{\|\vv(G,\FF)\|_{Y^n}}
   \le
   \sup_{\vv\in Y^n} \frac {b_e^n(\uu,\vv)}{\|\vv\|_{Y^n}}
   =
   \|\uu\|_{E^n}.
\end{align*}
This is \eqref{stab_gen} and finishes the proof of Theorem~\ref{thm_stab}.

\subsection{Proof of Theorem~\ref{thm_Cea}} \label{sec_pf_Cea}

For every time step $t_n$, let us define the DPG projection $\tilde\uu^n_h\in X^n_h$
of the exact solution at $t_n$, $\uu(t_n)=(u(t_n),\nabla u(t_n),\hat u(t_n),\hat\sigma(t_n))$
with obvious definitions of $\hat u(t_n)$ and $\hat\sigma(t_n)$, by
\[
   b_e^n(\tilde\uu^n_h,\vv) = b_e^n(\uu(t_n),\vv)\qquad\forall \vv\in Y^n_h,\ n=1,\ldots, N.
\]
Later, we will also need the first two components $\tilde u^n_h$ and $\tilde\ssigma^n_h$ of $\tilde\uu^n_h$.
Also note that there holds by Lemma~\ref{la_norms} and the best-approximation property of the DPG scheme
\begin{equation} \label{errnu}
   \mathrm{err}_n(u) := \|u(t_n)-\tilde u^n_h\|
   \le \|\uu(t_n)-\tilde\uu^n_h\|_{E^n} = \min_{\ww\in X^n_h} \|\uu(t_n)-\ww\|_{E^n}.
\end{equation}
We start by bounding $\|\tilde\uu^n_h-\uu^n_h\|_{E^n}$.
Relation \eqref{Enorm_discrete} means that
\[
   \|\tilde\uu^n_h-\uu^n_h\|_{E^n}^2 = b_e^n(\tilde\uu^n_h-\uu^n_h,\tilde\vv^n_h)
   \qquad\text{with}\quad
   \tilde\vv^n_h := (\tilde v^n_h,\tilde\ttau^n_h) := \Theta^n(\tilde\uu^n_h-\uu^n_h) \in Y^n_h.
\]
Recalling the definition of the extended bilinear form one notes that
\[
   b_e^n(\uu(t_n),\vv) = (f^n-\dot u(t_n)+\frac 1{k_n}u(t_n),v)\qquad\forall\vv=(v,\ttau)\in Y^n.
\]
Therefore,
\begin{align*}
   \|\tilde\uu^n_h-\uu^n_h\|_{E^n}^2
   &=
   b_e^n(\tilde\uu^n_h,\tilde\vv^n_h) - b_e^n(\uu^n_h,\tilde\vv^n_h)
   =
   b_e^n(\uu(t_n),\tilde\vv^n_h) - L_e^n(u^{n-1}_h;\tilde\vv^n_h)\\
   &=
   (f^n-\dot u(t_n)+\frac 1{k_n}u(t_n),\tilde v^n_h) - (f^n + \frac 1{k_n} u^{n-1}_h,\tilde v^n_h)\\
   &=
   (u(t_n)-u(t_{n-1}) - k_n\dot u(t_n), \frac 1{k_n}\tilde v^n_h)
   +
   (u(t_{n-1})-u^{n-1}_h,\frac 1{k_n}\tilde v^n_h)\\
   &\le
   \Bigl(\|u(t_n)-u(t_{n-1}) - k_n\dot u(t_n)\| + \|u(t_{n-1})-u^{n-1}_h\|\Bigr)
   \frac 1{k_n}\|\tilde v^n_h\|.
\end{align*}
By the selection of test norm \eqref{testnorm} and using relation \eqref{energy_ttt} we can bound
\[
   \frac 1{k_n}\|\tilde v^n_h\| \le \|\tilde\vv^n_h\|_{Y^n}
   = \|\Theta^n(\tilde\uu^n_h-\uu^n_h)\|_{Y^n} = \|\tilde\uu^n_h-\uu^n_h\|_{E^n}.
\]
Using also the bound
$\|\tilde u^n_h-u^n_h\|\le \|\tilde\uu^n_h-\uu^n_h\|_{E^n}$
by Lemma~\ref{la_norms} we have therefore shown that
\[
   \|\tilde u^n_h-u^n_h\|
   \le
   \|\tilde\uu^n_h-\uu^n_h\|_{E^n}
   \le
   \|u(t_n)-u(t_{n-1}) - k_n\dot u(t_n)\| + \|u(t_{n-1})-u^{n-1}_h\|.
\]
Applying repeatedly this estimation after bounding the last term like
\[
   \|u(t_{n-1})-u^{n-1}_h\| \le
   \|u(t_{n-1})-\tilde u^{n-1}_h\| + \|\tilde u^{n-1}_h-u^{n-1}_h\|
   =
   \mathrm{err}_{n-1}(u) + \|\tilde u^{n-1}_h-u^{n-1}_h\|
\]
we conclude that
\[
   \|\tilde\uu^n_h-\uu^n_h\|_{E^n}
   \le
   \|u_0-u^0_h\|
   +
   \sum_{j=1}^{n-1} \mathrm{err}_j(u)
   +
   \sum_{j=1}^n \|u(t_n)-u(t_{n-1}) - k_n\dot u(t_n)\|.
\]
Finally, the triangle inequality and bound \eqref{errnu} for $\mathrm{err}_j(u)$ show that
\begin{align*}
\lefteqn{
   \|\uu(t_n)-\uu^n_h\|_{E^n}
}\\
   &\le
   \|u_0-u^0_h\|
   +
   \sum_{j=1}^n \min_{\ww\in X^j_h} \|\uu(t_j)-\ww\|_{E^j}
   +
   \sum_{j=1}^n \|u(t_n)-u(t_{n-1}) - k_n\dot u(t_n)\|.
\end{align*}
This is the stated error estimate. The other two bounds are immediate by Lemma~\ref{la_norms}.
This finishes the proof of Theorem~\ref{thm_Cea}.

\section{Numerical experiments} \label{sec_num}

We present some numerical experiments for two problems in two dimensions, with
domain $\Omega=(0,1)\times (0,1)$.
Throughout we use constant in time, uniform triangular meshes $\cT_n=\cT$ with skeleton $\cS$,
and constant time steps $k_n=k=T/N$ with final time $T=0.1$. (There is no problem in selecting
larger final times, but then our model solutions are close to zero and would need to be
rescaled to provide reasonable tests.)

The trial spaces $X^n_h=X_h$ are
\begin{align*}
  X_h := P^0(\cT) \times \left[ P^0(\cT) \right]^2 \times S^1_0(\cS) \times P^0(\cS).
\end{align*}
Here, $P^0(\cT)$ and $P^0(\cS)$ indicate spaces of piecewise constant functions whereas
$S^1_0(\cS)$ is the space of continuous, piecewise linear functions on the skeleton with
zero trace on $\partial\Omega$.

The trial-to-test operator $\Theta^n=\Theta$ needed for the computation of optimal test functions
is approximated by solving, instead of \eqref{ttt}, the corresponding discrete problem with
enriched test space $\tilde Y^n_h=\tilde Y_h$ (instead of $Y^n$) that uses the same mesh
$\cT$ and piecewise polynomials of degree two (resp. three) for $v$ (resp. $\ttau$).
Our choice of enriched test space is based on the analysis
in~\cite{GopalakrishnanQ_14_APD} where the authors show that, for the Poisson equation, a
suitable strategy is to take the corresponding trial spaces and increase the polynomial
degrees by the space dimension $d=2$.
An analysis for singularly perturbed problems is open.

In the results below we use the following notation:
\begin{align*}
   \mathrm{err}(u) &= \|u(T)-u^N_h\|,\qquad
   \mathrm{err}(\ssigma) = \sqrt{k}\|\ssigma(T)-\ssigma^N_h\|,\\
   \mathrm{err}(\hat u) &= \Bigl(\|u(T) - \tilde u_h\|^2 + k\|\nabla u(T)-\nabla\tilde u_h\|^2\Bigr)^{1/2},\\
   \mathrm{err}(\hat\sigma) &=
          \sqrt{k}
          \Bigl(\|\ssigma(T) - \tilde\ssigma_h\|^2 + k\|\div\ssigma(T)-\div\tilde\ssigma_h\|^2\Bigr)^{1/2},\\
   \mathrm{err}(u_0) &= \|u_0-u^0_h\|,\qquad
   \mathrm{err}(u_\mathrm{ex}) = \|\uu(T)-\uu^N_h\|_{E^N}.
\end{align*}
Here, $\tilde u_h$ is the nodal interpolant in $H^1_0(\Omega)$ of $\hat u^N_h$ and by the definition
of the trace norm, $\mathrm{err}(\hat u)$ is an upper bound for $\|u(T) - \hat u^N_h\|_{1/2,\cS}$.
Similarly, $\tilde\ssigma_h$ is the lowest-order Raviart-Thomas interpolant of $\hat\sigma^N_h$ and,
therefore, $\sqrt{k}\|\hat\sigma(T)-\hat\sigma^N_h\|_{-1/2,\cS}\le \mathrm{err}(\hat\sigma)$.
These terms provide an upper bound for the error in $X^N_2$-norm, i.e., in the final time $T$:
\[
   \|\uu(T)-\uu^N_h\|_{X^N_2}^2
   \le
   \mathrm{err}(u)^2 + \mathrm{err}(\ssigma)^2 + \mathrm{err}(\hat u)^2 + \mathrm{err}(\hat\sigma)^2.
\]

\paragraph{Example 1.} We select the exact solution
\[
   u((x,y),t) = e^{-\pi^2t}\sin(\pi x)\sin(\pi y)
\]
so that $f=\pi^2 u$ and $u_0=u(\cdot,t=0)\in H^1_0(\Omega)$.
For the choice $k=\sqrt{h}/20$, Figure~\ref{fig1a} shows the corresponding errors and
confirms the prediction by Corollary~\ref{cor_conv}, that is, convergence order $O(h^{1/2})$.
The indicated slopes refer to $h$, not the number of elements.
Figure~\ref{fig1b} shows the ratios
\[
  \frac {(\|u^N_h\|^2+k\|\ssigma^N_h\|^2)^{1/2}}{\|u_0\|+k\sum_{n=1}^N \|f^n\|}
\]
which we expect to be bounded by $1$ by Theorem~\ref{thm_stab}.
The curve $S_1$ shows the values for $k = h/20$,
and $S_2$ for the previous selection of $k$.
In both cases the stability claim is confirmed.

\begin{figure}
\centerline{\includegraphics[width=0.7\textwidth,height=0.4\textwidth]{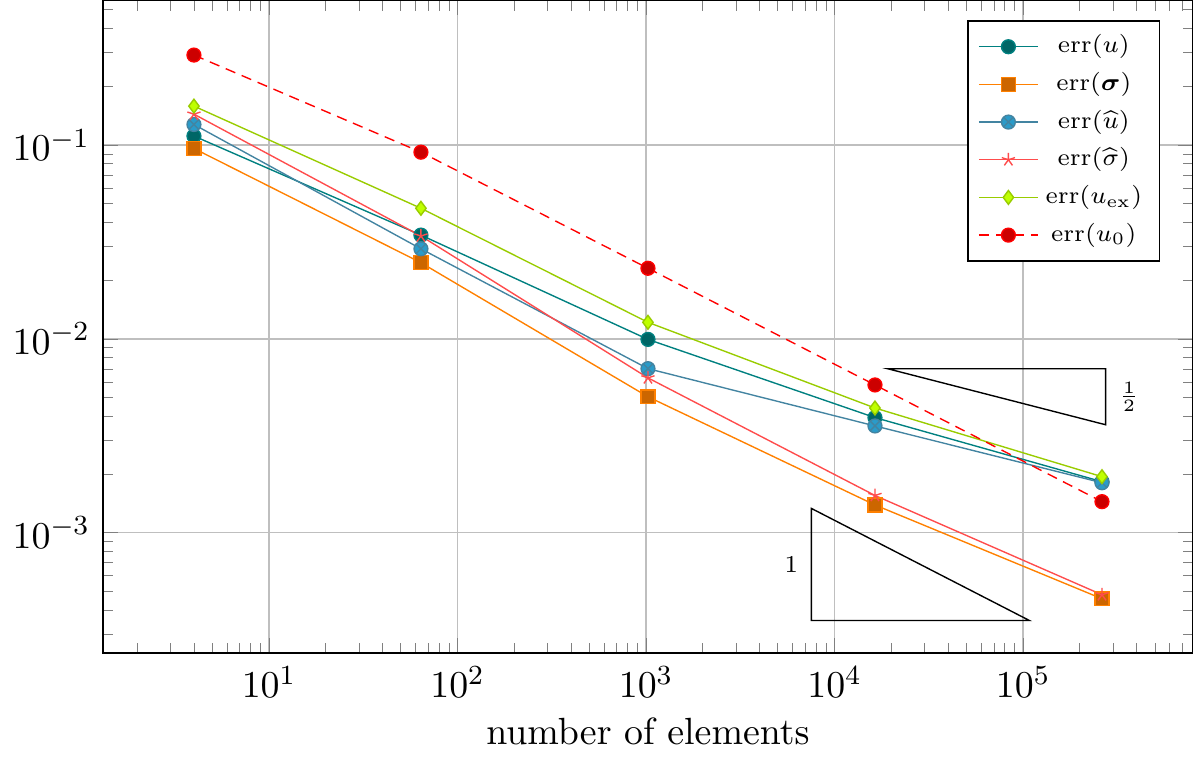}}
\caption{Errors for Example 1.}
\label{fig1a}
\end{figure}

\begin{figure}
\centerline{\includegraphics[width=0.7\textwidth,height=0.4\textwidth]{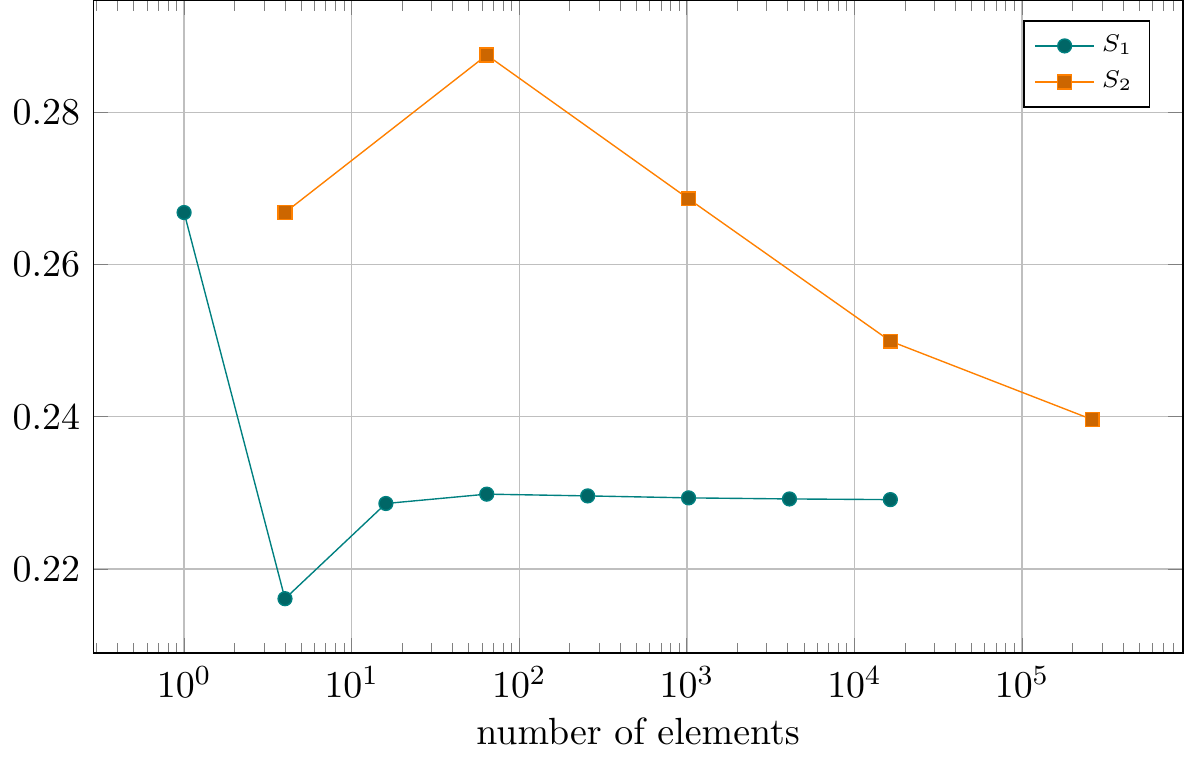}}
\caption{Testing stability for Example 1.}
\label{fig1b}
\end{figure}

\paragraph{Example 2.}
In this case, we test our method for a singular solution where the initial datum
does not satisfy the homogeneous boundary condition. We put
\[
   u_0(x,y) = (1-x) \sqrt{2} \sin (\pi y)
            = \frac {2\sqrt{2}}{\pi} \sin(\pi y) \sum_{j=1}^\infty \frac {\sin(j\pi x)}j
\]
and calculate $u$ by Fourier expansion:
\[
   u((x,y),t) = \frac {2\sqrt{2}}{\pi} \sin(\pi y)
                 \sum_{j=1}^\infty  e^{-(j^2+1)\pi^2 t} \frac {\sin(j\pi x)}j.
\]
For the numerical experiment we consider the first 1000 terms.
Figure~\ref{fig2} presents the results for the combination $k=\sqrt{h}/10$.
Despite of not fulfilling the regularity assumptions of Corollary~\ref{cor_conv}
we do observe convergence of order $O(h^{1/2})$.
As before, the indicated slopes refer to $h$.

\begin{figure}
\centerline{\includegraphics[width=0.7\textwidth,height=0.4\textwidth]{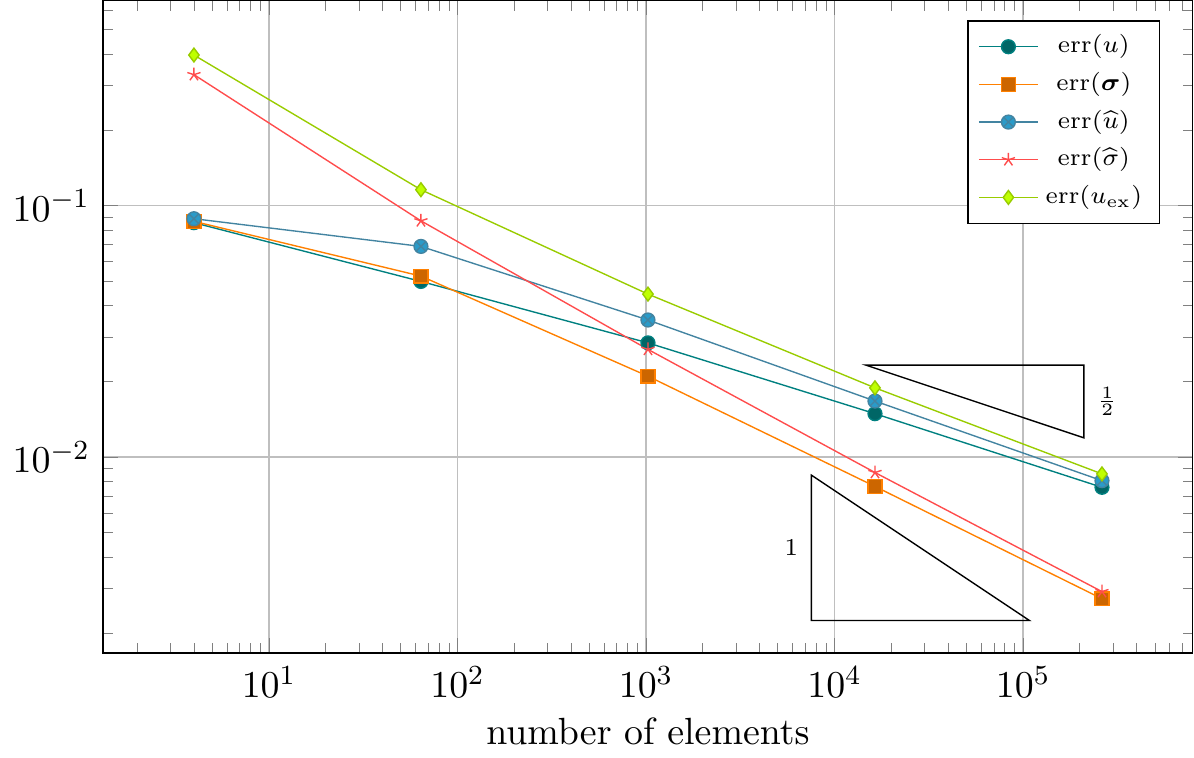}}
\caption{Errors for Example 2.}
\label{fig2}
\end{figure}

\bibliographystyle{siam}
\bibliography{/home/norbert/tex/bib/heuer,/home/norbert/tex/bib/bib}
\end{document}